\documentclass[11pt]{amsart}

\usepackage{eucal,amsrefs,graphicx,amssymb,mathrsfs}
\usepackage[all, knot]{xy}

\theoremstyle{definition}
\newtheorem{ex}{Example}[section]
\newtheorem*{ex*}{Example}

\newtheorem*{acknowledgement}{Acknowledgements}

\theoremstyle{plain}
\newtheorem{thm}{Theorem}[section]
\newtheorem{prop}[thm]{Proposition}
\newtheorem*{cor*}{Corollary}
\newtheorem{cor}[thm]{Corollary}
\newtheorem{lem}[thm]{Lemma}

\newtheorem*{thm*}{Theorem}

\theoremstyle{remark}
\newtheorem*{rem}{Remark}

\numberwithin{equation}{section}

\DeclareMathOperator{\Hom}{Hom}
\DeclareMathOperator{\Pic}{Pic}

\DeclareMathOperator{\CDiv}{CDiv}

\DeclareMathOperator{\Span}{span}
\DeclareMathOperator{\sgn}{sgn}
\DeclareMathOperator{\ad}{ad}
\DeclareMathOperator{\Bir}{Bir}
\DeclareMathOperator{\Vol}{Vol}

\def\A{\mathcal{A}}
\def\B{\mathcal{B}}
\def\C{\mathbb{C}}
\def\Q{\mathbb{Q}}
\def\P{\mathbb{P}}
\def\R{\mathbb{R}}
\def\Z{\mathbb{Z}}
\def\P{\mathbb{P}}

\def\N{\mathbb{N}}

\def\F{{\mathfrak{F}}}

\def\M{{\bf{M}}}

\def\SS{\mathcal{S}}

\def\i{{\texttt{\textit i}}}

\def\CDivT{\CDiv_T}
\def\fan{\Delta}

\def\n0{{\bf n_0}}

\providecommand{\pair}[2]{\langle#1,#2\rangle}
\def\transp #1{\vphantom{#1}^{\mathrm t}\! {#1}}

\begin{document}

\title[Three Dimensional Monomial Maps]
{On Degree Growth and Stabilization of Three Dimensional Monomial Maps}

\author{Jan-Li Lin}

\address{Department of Mathematics, Indiana University, Bloomington \\ IN 47405 \\ USA}

\email{janlin@indiana.edu}

\subjclass{}

\keywords{}

\begin{abstract}
In this paper, we develop several tools to study the degree growth and stabilization of monomial maps.
Using these tools, we can
classify semisimple three dimensional monomial maps by their dynamical behavior.
\end{abstract}

\maketitle

%\tableofcontents
%%%%%%%%%%%%%%%%%%%%%%%%%%%%%%%%%%%%%%%%%%%%%%%%%%%%%%%%%%%%%%%%%%%%%%
%%%%
%%%%   Section:  Introduction
%%%%
%%%%%%%%%%%%%%%%%%%%%%%%%%%%%%%%%%%%%%%%%%%%%%%%%%%%%%%%%%%%%%%%%%%%%%

\section{Introduction}

Given a rational selfmap $f:X\dashrightarrow X$ on an $n$-dimensional K\"{a}hler manifold $X$, one can define a pullback map
$f^*:H^{p,p}(X)\to H^{p,p}(X)$ for $0\le p\le n$. In general, the pullback does not commute with iteration, i.e.,
$(f^*)^k\ne(f^k)^*$. Following Sibony~\cite{S} (see also~\cite{FS}), we call the map $f$ (algebraically) \emph{stable} if 
the action on the cohomology of $X$ is compatible with iterations.
More precisely, $f$ is called {\em $p$-stable} 
if the pullback on $H^{p,p}(X)$ satisfies
$(f^*)^k=(f^k)^*$ for all $k\in\N$.

If $f$ is not $p$-stable on $X$, one might try to find a birational change of coordinate $h: X'\dashrightarrow X$
such that $\tilde{f} = h^{-1}\circ f\circ h$ is $p$-stable. This is not always possible even for $p=1$, as shown by
Favre~\cite{Fa}. However, for $p=1$ and $n=2$, one can find such a stable model (with at worst quotient singularities)
for quite a few classes of surface maps \cite{DF,FJ}.
Also for $p=1$, such model can be obtained for certain monomial maps \cite{Fa,JW,ljl}.

For an $n\times n$ integer matrix $A=(a_{i,j})$, the associated monomial map $f_A:(\C^*)^n\to(\C^*)^n$ is defined by
\[
\textstyle{f_A(x_1,\cdots,x_n)=( \prod_j x_j^{a_{1,j}},\cdots ,\prod_j x_j^{a_{n,j}} ).  }
\]
The morphism $f_A$ extends to a rational map, also denoted $f_A$, on any $n$-dimensional toric variety. The question of
finding a stable model for $f_A$ (or showing there is no stable model for certain $f_A$) has been studied in \cite{Fa,JW,ljl}.
In particular, for dimension two, the stabilization problem if fully classified in \cite{Fa} and \cite{JW}.

In this paper, we focus on the case when $n=3$ and $A$ is diagonalizable.
We deal both the $1$-stable and the $2$-stable problems. There are more cases than
dimension two. A main case where a model that is both $1$-stable and $2$-stable can be obtained by
performing proper modification is summarized in the following theorem.

\begin{thm}
\label{thm:main_1}
Let $\fan$ be a fan in $N\cong \Z^3$, and $f_A: X(\fan) \dashrightarrow X(\fan)$ be the monomial map
associated to $A$.
If $A$ is diagonalizable, and for each eigenvalue $\mu$ of $A$, $\mu/\bar{\mu}$ is a root of unity.
Then
there exists a complete simplicial refinement $\fan'$ of $\fan$ and $k_0\in\N$ such that
the map $f_A^k: X(\fan') \dashrightarrow X(\fan')$ is both 1-stable and 2-stable
for all $k\ge k_0$.
\end{thm}

We note that for some subcases of the above case, even better results will hold.
For example, if the eigenvalues have different modulus, then we can make $X(\fan')$ smooth and projective.
Another particularly nice subcase is when
the set $\{A^k|k\in\N\}$ is a finite set, then in fact we can find a smooth projective
$X(\fan')$ such that $f_A$ is an automorphism on $X(\fan')$. This is related to the resolution of indeterminacy of pairs
for birational maps. See sections~\ref{sec:stab}--\ref{sec:proof_main_thm} for more details.

The case where there are two complex eigenvalues $\mu,\bar{\mu}$ of $A$ with $\mu/\bar{\mu}$ not a root of unity is more complicated.
We discuss this case in section~\ref{sec:complicated_cases}. Nevertheless, this case contains several interesting subcases.
For instance, the following phenomenon can happen.
\begin{itemize}
\item We can have the case where $f_A$ being $1$-stable on a smooth projective variety, but does not have a $2$-stable model, and vice versa.
\item There is also the case where given a toric variety $X$, there is a birational model for $f_A$ that is stable, but $f_A$
       cannot be made stable by performing the blowup process on $X$.
\item It is also possible to have $f_A$ which has no $1$-stable model and no $2$-stable model.
\end{itemize}

The paper is organized as follows. In section~\ref{sec:toric}, we review known facts about toric varieties and monomial maps, then
we developed some technical tools about stabilization and the degree sequence in section~\ref{sec:stab}. A case which is related
to the resolution of indeterminacy of pairs is studied in section~\ref{sec:res_indet}. Theorem~\ref{thm:main_1} is then proved in
section~\ref{sec:proof_main_thm}. Finally, the more complicated case where
there are two complex eigenvalues $\mu,\bar{\mu}$ of $A$ with $\mu/\bar{\mu}$ not a root of unity,
is discussed in section~\ref{sec:complicated_cases}.

\begin{acknowledgement}
The work is completed while the author was visiting the Institute for Computational and Experimental Research in Mathematics (ICERM).
The author would like to thank ICERM for hospitality.
I would also like to thank Elizabeth Wulcan for helpful discussions.
\end{acknowledgement}

%%%%%%%%%%%%%%%%%%%%%%%%%%%%%%%%%%%%%%%%%%%%%%%%%%%%%%%%%%%%%%%%%%%%%%
%%%%
%%%%   Section: Preliminaries on Toric Varieties
%%%%
%%%%%%%%%%%%%%%%%%%%%%%%%%%%%%%%%%%%%%%%%%%%%%%%%%%%%%%%%%%%%%%%%%%%%%

\section{Preliminaries on Toric Varieties}
\label{sec:toric}

A toric variety is a (partial) compactification of the torus $T\cong (\C^*)^m$, which contains $T$ as a dense subset and which admits an action of $T$ that extends the natural action of $T$ on itself. We briefly recall some of the basic definitions. All results stated in this section are known,
so the proofs are omitted. We refer the readers to~\cite{Fu} and~\cite{O} for details about toric varieties.

\subsection{Fans and toric varieties}

Let $N$ be a lattice isomorphic to $\Z^n$ and let $M=\Hom(N,\Z)$
denote the dual lattice. The algebraic torus $T=T_N\cong (\C^*)^n$ is canonically identified
with the group $\Hom_\Z(M,\C^*)$.
Set $N_\Q:=N\otimes_\Z \Q$,
$N_\R:=N\otimes_\Z \R$, and define $M_\Q$ and $M_\R$ analogously. Let $\R_+$ and $\R_-$ denote the sets of
non-negative and non-positive numbers, respectively.

A \emph{convex rational polyhedral cone} $\sigma$ is of the form $\sigma=\sum\R_+v_i$ for some $v_i\in N$.
We will simply say that $\sigma$ is a cone \emph{generated}
by the vectors $v_i$.
If $\sigma$ is convex and does not contain any line in $N_\R$ it is said to be \emph{strictly convex}.
A \emph{face} of $\sigma$ is the intersection of
$\sigma$ and a supporting hyperplane. The \emph{dimension} of $\sigma$
is the dimension of the linear space spanned by
$\sigma$. One-dimensional cones are called \emph{rays}. Given a ray
$\gamma$, the associated \emph{ray generator} is the first nonzero lattice
point on $\gamma$.
A $k$-dimensional cone is \emph{simplicial}
if it can be generated by $k$ vectors. A cone is \emph{regular} if it
is generated by part of a basis for $N$.

A \emph{fan} $\fan$ in $N$ is a finite collection of rational
strongly convex cones in $N_\R$ such that each face of a cone in
$\fan$ is also a cone in $\fan$, and the intersection of
two cones in $\fan$ is a face of each of them. Let $\fan(k)$ denote
the set of cones in $\fan$ of dimension $k$.
A fan $\fan$ determines a toric variety $X(\fan)$ by patching together affine
toric varieties $U_\sigma$ corresponding to the cones
$\sigma\in\fan$.
The \emph{support} $|\fan|$ of the fan $\fan$ is the union of all
cones of $\fan$. In fact, given any collection of cones $\Sigma$, we
will use $|\Sigma|$ to denote the union of the cones in $\Sigma$.
If $|\fan|=N_\R$, then the fan $\fan$ is said to be \emph{complete}.
If all cones in $\fan$ are simplicial then $\fan$ is said to be
\emph{simplicial}, and if all cones are regular, $\fan$ is said to
be \emph{regular}.
A fan $\fan'$ is a \emph{refinement} of $\fan$ if for each cone $\sigma'$ in $\fan'$
there is a cone $\sigma\in\fan$ such that $\sigma'\subset\sigma$.

A toric variety $X(\fan)$ is compact if and only if $\fan$ is
complete. If $\fan$ is simplicial, then $X(\fan)$ has at worst quotient
singularities, i.e., $X(\fan)$ is an orbifold. $X(\fan)$ is smooth (non-singular) if and only if $\fan$
is regular.
For any fans $\fan_1$ and $\fan_2$ in $N$ there is a common refinement so
that there are resolutions of singularities $X(\fan')\to
X(\fan_j)$; in particular $X(\fan_1)$ and
$X(\fan_2)$ are birationally equivalent.

\subsection{Monomial maps and the condition for being stable}

Suppose $A:N\to N'$ is a homomorphism of lattices, $\fan$ is a fan in $N$, and $\fan'$ is a fan in $N'$.
A homomorphism of lattices $A:N\to N'$ induces a group homomorphism $f_A: T_{N}\to T_{N'}$ which is
given by monomials on each coordinate. One can extend $f_A$ to an {\em equivariant rational map}
$f_A:X(\fan)\dashrightarrow X(\fan')$. On a complete toric variety, $f_A$ is {\em dominant} if and only if
$A_\R=(A\otimes\R):N_\R\to N'_\R$ is surjective. The map $f_A$ is called {\em semisimple} if $A$
is diagonalizable.
Given a cone $\sigma\in\fan$, we say that $\sigma$ {\em maps regularly} to $\fan'$ by $A$ if there is a cone
$\sigma'\in\fan'$ such that $A(\sigma)\subseteq\sigma'$. In this case, we call the smallest such cone in $\fan'$
the {\em cone closure} of the image of $\sigma$, and denote it by $\overline{A(\sigma)}$.

For a complete toric variety $X(\fan)$ associated to a complete fan $\fan$, the group of torus invariant Cartier divisors
on $X(\fan)$ is denoted by $\CDiv_T(X(\fan))$, and the {\em Picard group} denoted by $\Pic(X(\fan))$. Given a
monomial map $f_A:X(\fan)\dashrightarrow X(\fan')$ we can define a {\em pullback map}
$f_A^*:\Pic(X(\fan))\to \Pic(X(\fan'))$, as well as the pullback map
$f_A^*:\CDiv_T(X(\fan))\to \CDiv_T(X(\fan'))$.
A toric rational map $f_A:X(\fan)\dashrightarrow X(\fan)$ is {\em strongly $1$-stable} if
$(f_A^k)^*=(f_A^*)^k$ as maps of $\CDivT(X(\fan))_\Q$ for all $k\in\N$. It is {\em $1$-stable} if
$(f_A^k)^*=(f_A^*)^k$ as maps of $\Pic(X(\fan))_\Q$, for all $k$.
On a projective toric variety, $f_A$ is strongly 1-stable if and only if it is 1-stable.
The following theorem gives a geometric condition on being strongly $1$-stable.

\begin{thm}\label{thm:sas}
A toric rational map $f_A:X(\fan)\dashrightarrow X(\fan)$ is strongly $1$-stable if and only if
for all ray $\tau\in\fan(1)$ and for all $n\in\N$, $\overline{A^n(\tau)}$ maps regularly to $\fan$ by $A$.
\end{thm}

For the proof of Theorem~\ref{thm:sas}, as well as more details on the $1$-stability of monomial maps, see
\cite{JW,ljl}. Notice that what we call $1$-stable here is called (algebraically) stable in these paper.

%%%%%%%%%%%%%%%%%%%%%%%%%%%%%%%%%%%%%%%%%%%%%%%%%%%%%%%%%%%%%%%%%%%%%%
%%%%
%%%%   Section: General results on Stabilization
%%%%
%%%%%%%%%%%%%%%%%%%%%%%%%%%%%%%%%%%%%%%%%%%%%%%%%%%%%%%%%%%%%%%%%%%%%%

\section{Degrees and Stabilization}
\label{sec:stab}

In this section, we show various results about the degrees and stabilization of monomial maps.
Although they serve as tools to prove Theorem~\ref{thm:main_1}, these results hold not only in
dimension three, but, in many cases, in arbitrary dimensions.

\subsection{Duality between $p$ and $(n-p)$}

For $A\in M_n(\Z)$, let
\[
A'=|\det(A)|\cdot A^{-1} = \sgn(\det (A))\cdot\ad(A),
\]
where $\sgn(.)$ is the sign function,
and $\ad(A)$ is the classical adjoint matrix of $A$.
Notice that $A'\in M_n(\Z)$, and if $\det(A)\ne 0$, then
$\det(A')=\det(A)^{n-1}$ is also nonzero.

\begin{prop}
For a monomial map $f_A$, its pullback map $f_A^*$ on $H^{p,p}$ is adjoint, up to a scaler,
to the pullback map $f_{A'}^*$ on $H^{n-p,n-p}$ under the intersection pairing.

More precisely, let $\pair{\,}{}$ be the intersection pairing, $\alpha\in H^{p,p}(X(\fan))$, $\beta\in H^{n-p,n-p}(X(\fan))$,
then we have
\[
\pair{f_A^*\alpha}{\beta} = |\det(A)|^{p-n+1}\cdot \pair{\alpha}{f_{A'}^*\beta}.
\]
\end{prop}

\begin{proof}
First, assume that $X(\fan)$ is simplicial and projective. Under this assumption, $H^{*,*}(X(\fan))\cong A^*(X(\fan))$
is generated by $\Pic(X(\fan))_\R$. Moreover, every divisor in a projective variety is a difference of two ample divisors.
Therefore, it suffices to prove the equation for products of ample divisors of $X(\fan)$.
Recall the following diagram for toric rational map $f_A$.
\[
\xymatrix{
& X(\tilde{\fan})\ar[ld]_{\pi}\ar[rd]^{\tilde{f}_A}\\
X(\fan)\ar@{-->}[rr]_{f_A} & & X(\fan)
}
\]

Let $\alpha = [D_1] . \cdots . [D_p]$ and $\beta = [D_{p+1}] . \cdots . [D_n]$, where $D_i$ is an ample divisor with associated
polytope $P_i$, $i=1,\cdots,n$. Then, as a consequence of the Riemann-Roch theorem for toric varieties 
(see \cite[pp.116--117]{Fu}), we have
\[
\begin{split}
\pair{f_A^*\alpha}{\beta}
&= \pi_*\tilde{f}_A^*(D_1 . \cdots . D_p).(D_{p+1} . \cdots . D_n) \\
&= \tilde{f}_A^*(D_1 . \cdots . D_p).\pi^*(D_{p+1} . \cdots . D_n) \\
&= n!\cdot MV(\transp{A} P_1,\cdots,\transp{A} P_p, P_{p+1},\cdots, P_n) \\
&= n!\cdot |\det(\transp{A})|\cdot MV(P_1,\cdots,P_p, \transp{A}^{-1} P_{p+1},\cdots, \transp{A}^{-1} P_n) \\
&= |\det(\transp{A})|^{p-n+1}\cdot n!\cdot MV(P_1,\cdots,P_p, \transp{A}' P_{p+1},\cdots, \transp{A}' P_n) \\
&= |\det(\transp{A})|^{p-n+1}\cdot \pair{\alpha}{f_{A'}^*\beta}
\end{split}
\]
Here $MV(\cdots)$ denotes the mixed volume of polytopes.

Finally, for a general toric variety $X(\fan)$, we can subdivide $\fan$ to obtain a refinement $\fan'$ such that $X(\fan')$
is simplicial and projective. The induced map $A^*(X(\fan))\to A^*(X(\fan'))$ will be injective, and the formula for $X(\fan)$
follows from the result for $X(\fan')$.
\end{proof}
We list two consequences of the proposition.

\begin{cor}
\label{cor:dual_stable}
The map $f_A$ is $p$-stable if and only if $f_{A'}$ is $(n-p)$-stable.
\end{cor}

\begin{proof}
Notice that
\[
\begin{split}
  &\  (f_A^*)^k = (f_A^k)^* \\
 \Longleftrightarrow &\  \pair{(f_A^*)^k\alpha}{\beta} = \pair{(f_A^k)^*\alpha}{\beta},
        \text{ $\forall \alpha\in H^{p,p}$, $\beta\in H^{n-p,n-p}$} \\
 \Longleftrightarrow &\  \pair{\alpha}{(f_{A'}^*)^k\beta} = \pair{\alpha}{f_{(A^k)'}^*\beta}\\
 \Longleftrightarrow &\  (f_{A'}^*)^k = (f_{A'}^k)^*.
\end{split}
\]
The last implication is because we have $(A^k)'=(A')^k$ for all $k$. Thus the corollary is proved.
\end{proof}

Given an ample divisor $D$ on a projective toric variety $X(\fan)$, we define the $p$-th degree of $f_A$ with respect to $D$,
denoted $\deg_{D,p}(f_A)$, as
\[
\deg_{D,p}(f_A) = \pair{f_A^* D^p}{D^{n-p}}.
\]
We have the following relation between the $p$-th degree and the $(n-p)$-th degree of a monomial map with respect to $D$, 
since the intersection pairing is symmetric.

\begin{cor}
For any ample divisor $D$, we have
\[
\deg_{D,p}(f_A) = |\det(A)|^{p-n+1} \deg_{D,n-p}(f_{A'}).
\]
\end{cor}

In particular, for $p=n-1$, we get
$\deg_{D,n-1}(f_A) = \deg_{D,1}(f_{A'})$.

\subsection{Stability and linear recurrence of the degree sequence}

The following property shows the relation between $p$-stable and the $p$-th degree satisfying
a linear recurrence relation.
This is probably well known to the experts. We include the result here for completeness.

\begin{prop}
\label{prop:deg_rec}
Given a projective simplicial toric variety $X=X(\fan)$,
suppose that the monomial map $f_A$ is $p$-stable on $X$.
Then for any $T_N$-invariant ample divisor $D$ of $X$, the degree sequence
$\{\deg_{D,p}(f_A^k)\}_{k=1}^\infty$ satisfies a linear recurrence relation.
\end{prop}

\begin{proof}
Since $D^n = n!\Vol(P_D)>0$ for ample divisor $D$, we know the cohomology class $[D^p]$ is nonzero
in $H^{p,p}(X)$, we can extend $[D^p]$ to form a
basis $\B = \{ [D^p] , \theta_1, \cdots, \theta_r \}$ for $H^{p,p}(X)$ such that
$D^{n-p} . \theta_i = 0$ for all $i=1,\cdots,r$.
Then $\deg_{D,p}(f_A)$ is the $(1,1)$-entry of the matrix $\A$ for $f_A^*$ with respect to $\B$.
Since $f_A$ is $p$-stable, i.e., $(f_A^k)^* = (f_A^*)^k$ for all $k\in\N$, we know that
$\deg_{D,p}(f_A^k)$ is the $(1,1)$-entry of the matrix $\A^k$. Therefore, by the
Cayley-Hamilton theorem, the sequence
$\{\deg_{D,p}(f_A^k)\}_{k=1}^\infty$ satisfies the linear recurrence relation induced by the
characteristic polynomial of $\A$.
\end{proof}

However, in practice, it is usually difficult to make a birational change of coordinate to make
the map stable, even for $p=1$. Instead, we impose the following slightly weaker condition. This is
a ``stable version of $p$-stable'', thus we write $p$-SS (stably stable) for the condition.

\begin{itemize}
\item[($p$-SS)]
There exists an integer $k_0 \ge 1$, such that for all $k,l\ge k_0$, we have $(f^{k+l})^* = (f^k)^*\circ (f^l)^*$.
\end{itemize}

\begin{rem}
We have the following consequences of $p$-SS.
\begin{enumerate}
\item For all $k\ge k_0$, the map $f^k$ is $p$-stable.
\item The degree sequences $\{\deg_{D,p}(f_A^{kj+l})\}_{j=1}^\infty$ satisfy a (same) linear recurrence relation for all $k,l\ge k_0$.
The proof is similar to the proof for Proposition~\ref{prop:deg_rec}.
Thus, the sequence $\{\deg_{D,p}(f_A^k)\}_{k=1}^\infty$ also satisfies a linear recurrence relation.
\end{enumerate}
\end{rem}

\subsection{Simultaneous stabilization}

Given an $n\times n$ matrix $A\in M_n(\R)$ with $\det(A)\ne 0$,
let $\mu_1,\cdots,\mu_n$ be the eigenvalues of $A$, counting multiplicities.
We say that $A$ has {\em isolated spectrum} if $\mu_i\ne \mu_j$ for $i\ne j$, and that
$A$ has {\em absolutely isolated spectrum} if $|\mu_i|\ne |\mu_j|$ for $i\ne j$.

\begin{thm}
\label{thm:JW}
Let $\fan$ be a fan in $N\cong \Z^n$, and $A_1,\cdots,A_m\in M_n(\Z)$ be matrices with absolutely isolated spectrum.
Then there exists a complete refinement $\fan'$ of $\fan$ such that $X(\fan')$ is smooth and
projective and the induced maps $f_{A_i}: X(\fan') \dashrightarrow X(\fan')$ are $1$-SS for all $i=1,\cdots,m$.
\end{thm}

The $m=1$ case of the theorem is proved in \cite[Theorem A']{JW}. The proof here basically follows theirs, with a slight
modification. So we briefly recall their proof first.

In \cite{JW}, the authors introduced the notion of {\em adapted systems of cones}. It is a collection of simplicial cones
satisfying certain conditions. We list some properties of adapted systems of cones in the following. For details,
see \cite[Section 4]{JW}.

\begin{lem}\label{JW45}\cite[Lemma 4.5]{JW}
Let $\SS = \{\sigma(V,\eta)\}$ be an adapted system of cones and $v\in N$. Then there exists $k_0\in\N$ such that for
$k\ge k_0$, $A^k(v)\in\sigma(V,\eta)$ for some $\sigma(V,\eta)\in\SS$.
\end{lem}

\begin{lem}\label{JW46}\cite[Lemma 4.6]{JW}
Let $\SS = \{\sigma(V,\eta)\}$ be an adapted system of cones and $v\in N$. Then there exists $k_0\in\N$ such that for
$k\ge k_0$, $\SS$ is invariant under $A^k$.
\end{lem}

\begin{lem}\label{JW411}\cite[Lemma 4.11]{JW}
Let $\fan$ ba a fan in $N$ that contains an adapted system of cones, and let $\fan'$ be a refinement of $\fan$.
Assume that for every invariant rational subspace $V$ of $N_\R$, there is a subfan of $\fan'$ with support $V$.
Then $\fan'$ contains a unique adapted system of cones.
\end{lem}

\begin{lem}\label{JW412}\cite[Lemma 4.12]{JW}
There exists a rational adapted system of cones for every $A\in M_n(\Z)$ with absolutely isolated spectrum.
\end{lem}

\begin{proof}[Proof of Theorem~\ref{thm:JW}]
First, we refine the fan $\fan$ so that for each rational invariant subspace $V$ of each $A_i$, there is a
subfan of $\fan$ whose support is $V$.

Next, we refine $\fan$ so that for each $A_i$, it contains an adapted system of cones. We achieve this goal by
doing the following. Let $\SS_i$ be a rational adapted system of cones for $A_i$; its existence is guaranteed by
Lemma~\ref{JW412}. Let $\fan_i$ be a complete fan generated by cones in $\SS_i$. Now take a common refinement
of $\fan$ and all $\fan_i$, then take a further refinement to make the fan both
regular and projective. Call this refined fan $\fan'$. By Lemma~\ref{JW411},
$\fan'$ contains a unique adapted system of cones $\SS'_i$ for each $A_i$.

Finally, by Lemma~\ref{JW45}, there is a number $k_0$ such that for all $k\ge k_0$ and $\gamma\in\fan'(1)$,
we have $A_i^k(\gamma)\subset \sigma$ for some $\sigma \in\SS'_i$. Furthermore, by Lemma~\ref{JW46}, we can
choose an even larger $k_0$ such that for all $k\ge k_0$, each $\SS'_i$ is invariant under $A_i^k$.
These conditions will imply that each $A_i$ is $1$-SS, which concludes our proof.
\end{proof}

\begin{cor}
\label{cor:stables}
Let $\fan$ be a fan in $N\cong \Z^n$, and $A\in M_n(\Z)$ be a matrix with absolutely isolated spectrum.
Then there exists a complete refinement $\fan'$ of $\fan$ such that $X(\fan')$ is smooth and
projective and the induced maps $f_A: X(\fan') \dashrightarrow X(\fan')$ is both $1$-SS and
$(n-1)$-SS.
\end{cor}

\begin{proof}
By Corollary~\ref{cor:dual_stable}, applying $A$ and $A'$ to Theorem~\ref{thm:JW}, the corollary then follows.
\end{proof}

\subsection{Stabilization in a special case}

The following is a case we will encounter when we prove Theorem~\ref{thm:main_1}, we prove a version of it
that is valid in all dimensions.

\begin{prop}
\label{prop:two_real_pos_eval}
Let $\fan$ be a fan in $N\cong \Z^n$, and $A\in M_n(\Z)$ be a diagonalizable matrix
whose eigenvalues are two distinct real numbers
$\mu_1, \mu_2 \in \R$ (each with some multiplicities) such that $|\mu_1|\ne |\mu_2|$.
Then there exists a complete refinement $\fan'$ of $\fan$ such that $X(\fan')$ is simplicial and
projective and the induced maps $f_A: X(\fan') \dashrightarrow X(\fan')$ is both $1$-stable and
$(n-1)$-stable.
\end{prop}

\begin{proof}
For each ray $\gamma\in\fan(1)$, consider its ray generator $v$.
Write $v = v_1 + v_2$, where $v_i$ is an eigenvector of $\mu_i$, $i=1,2$.
If one of the $v_i$ is zero, then $\gamma$ is in fact invariant under $A$.
Now suppose both $v_1$ and $v_2$ are nonzero. Let $\Gamma_\gamma := \Span\{v_1,v_2\}$
be the $2$-dimensional vector space spanned by $v_1,v_2$, then $\Gamma_\gamma$ is invariant
under $A$, i.e., $A(\Gamma_\gamma)=A^{-1}(\Gamma_\gamma)=\Gamma_\gamma$,
and $A^k(\gamma)\subset\Gamma_\gamma$ for all $k\in\Z$.

Notice that $\Gamma_\gamma\cap N$ has rank $2$, since both $v, A(v)\in \Gamma_\gamma\cap N$, and they are independent.
Thus, $\Gamma_\gamma\cap\fan := \{ \Gamma_\gamma\cap\sigma | \sigma\in\fan \}$ is a fan of rational polyhedral cones with
support $\Gamma_\gamma$. We can refine $\Gamma_\gamma\cap\fan$ to obtain a refinement $\fan_{\gamma}$
such that under the map $A|_{\Gamma_\gamma}$, each ray in $\fan_\gamma$ either maps to another ray, or maps into
a two dimensional cone $\sigma$ of $\fan_\gamma$ with $A(\sigma)\subset\sigma$. Moreover, notice that
we can in fact find a refinement $\fan_\gamma$ such that
the same requirement holds for $A'$ too.

Now take a common simplicial refinement of $\fan$ and $\fan_\gamma$, $\gamma\in\fan(1)$ without adding rays,
we obtain a new fan $\fan'$.
We have $\fan'(1) = \cup_{\gamma\in\fan(1)} \fan_\gamma (1) $, thus under the map $A$,
each ray in $\fan'$ is either invariant, or eventually maps to some two dimensional cone which is invariant.
Thus $f_A$ is $1$-stable. Similarly, we know $f_{A'}$ is $1$-stable, thus $f_A$ is $(n-1)$-stable too.
\end{proof}

More generally, now we assume that $A$ is diagonalizable, and
there exist $r_1>r_2>0$ such that for each eigenvalue $\mu$ of $A$, $\mu/\bar{\mu}$
is a root of unity, and either $|\mu|=r_1$ or $|\mu|=r_2$. Then there exists some
$\ell$ such that $A^\ell$ satisfies the assumption of Proposition~\ref{prop:two_real_pos_eval}.
Observe that for any ray $\gamma$ with generator $v$, let $\Gamma_\gamma$ be the plane spanned by $v$
and $A^\ell(v)$, then the orbit $\{ A^k(v) | k\in\Z\}$ is contained in finitely many planes
$A^k(\Gamma_\gamma)$, $k=0,\cdots, \ell-1$. We can then simultaneously refine
the fans $A^k(\Gamma_\gamma)\cap \fan$, such that under $A$, each ray in the subdivided fan
either maps into another ray, or maps into a two dimensional cone $\sigma_0$ and
there exist two dimensional cones $\sigma_1,\cdots,\sigma_\ell=\sigma_0$ in the refined cones such that
$A(\sigma_k)\subset \sigma_{k+1}$ for $k=0,\cdots,\ell-1$. This implies $A^\ell(\sigma_k)\subset\sigma_k$.
We can also achieve the same requirement for $A'$ simultaneously, as in the proof of Proposition~\ref{prop:two_real_pos_eval}.
Now take a common refinement of $\fan$ with all these refined fans, without adding more rays, we
obtain a new fan $\fan'$ such that $f_A$ is both $1$-stable and $(n-1)$-stable on $X(\fan')$.

\section{Resolution of Indeterminacy of Pairs for Toric Varieties}
\label{sec:res_indet}

Let $X$ be a projective variety and $G$ be a finite subgroup of the group of birational transformations
$\Bir(X)$. A resolution of indeterminacy
of the pair $(X,G)$ consists of a smooth variety $X'$, birationally
equivalent to $X$, and a birational map $\pi : X'\to X$ such that for every
$g\in G$ the composite map $\pi^{-1} g \pi$ is an automorphism of $X$.

In a paper of de Fernex and Ein \cite{Ein},
the authors show that in characteristic zero the resolution of indeterminacy of a pair $(X,G)$
always exists. They also obtain explicit construction of the resolution in some two dimensional cases.
Also, in Chel{\cprime}tsov's paper \cite{Che}, an explicit construction
of resolution of indeterminacy of pairs is given in dimension three using the minimal model program.

In this section, we first mention a known proposition which implies the resolution of indeterminacy of pairs
in the case of toric varieties and equivariant birational maps. Then we will modify the condition slightly,
obtaining a theorem which can be applied to a case in our classification.

First, let us recall the following.

\begin{prop}[\cite{Br,THS}]
Let $N$ be a lattice and $\fan\cong \Z^n$ be a fan of $N_\R$. Let $G$ be a finite group of automorphisms of $N$.
Then there exists a refinement $\fan'$ of $\fan$ which is smooth, projective, and invariant by $G$.
\end{prop}

In \cite{Br, THS}, the authors use the above result to show the existence of smooth
projective compactification of an algebraic torus over a scheme.
The proposition also gives an explicit way to construct the resolution of indeterminacy of pairs
in the case of toric varieties. Moreover, this explicit resolution works for all characteristic.

\begin{thm}[Resolution of Indeterminacy of Pairs]
Let $X(\fan)$ be a toric variety and $G$ be a finite group of equivariant birational maps of $X(\fan)$.
Then there exists a smooth projective toric variety $X(\fan')$ and a birational morphism $\pi: X(\fan')\to X$
such that for each $g\in G$ the composite map $\pi^{-1} g \pi$ is an equivariant automorphism of $X(\fan')$.
In other words, elements in $G$ are birational conjugate to automorphisms of $X(\fan')$.
\hfill\qed
\end{thm}

However, to suit our general purpose, we need a different condition. Recall that the set of dominant monomial maps
of $n$-dimensional toric varieties corresponds to the set of integer matrices with nonzero determinant:
\[
M_n(\Z)\cap GL_n(\Q) = \{ A\in M_n(\Z) | \det(A)\ne 0\}.
\]
Consider matrices of the form $d\cdot I_n$ with $d>0$, every cone is invariant under the action of $d\cdot I_n$.
Thus every fan is invariant, too. Furthermore, denote $\Q^+$ as the set of positive rational numbers
and identify $\Q^+$ with $\Q^+\cdot I_n$. Let $\F$ be the set of complete rational fans in $N_\R$. Then
$GL_n(\Q)$ acts on $\F$, and
\[
\Q^+ = \bigcap_{\fan\in\F} GL_n(\Q)_\fan,
\]
where $GL_n(\Q)_\fan$ is the stabilizer of $\fan\in\F$ in $GL_n(\Q)$.

Let $p:GL_n(\Q)\to PGL^+_n(\Q):= GL_n(\Q) / \Q^+$ be the projection map.

\begin{prop}
\label{thm_res_fan}
Let $N$ be a lattice and $\fan\cong \Z^n$ be a fan of $N_\R$. Let $G$ be a submonoid of $M_n(\Z)\cap GL_n(\Q)$.
If $p(G)$ is finite,
then there exists a refinement $\fan'$ of $\fan$ which is projective and invariant by $G$.
\end{prop}

\begin{proof}
First, take all the hyperplanes in $N_\R$ which contain some of the $(n-1)$-dimensional cones of $\fan$.
These hyperplanes determine a complete refinement $\fan_1$ of $\fan$. Moreover, the corresponding variety $X(\fan)$ is
projective (see \cite[Proposition 2.17]{O}).

Then we consider the fan
\[
\fan' = \bigcap_{g\in G} g \fan_1,
\]
where $g \fan_1 = \{ g\sigma | \sigma\in\fan_1\}$. The intersection is in fact finite, because $p(G)$ is a finite set.
Thus $\fan'$ is a finite fan.
Moreover, the fan $\fan'$ is invariant under $G$, and it is also projective since it contains all the hyperplanes
spanned by its $(n-1)$-dimensional cones.
\end{proof}

One can translate the above proposition into the following toric version.

\begin{prop}
\label{thm_res_morphism}
Let $X(\fan)$ be a toric variety and $G$ be a submonoid of $M_n(\Z)\cap GL_n(\Q)$ with $p(G)$ finite.
Then there exists a projective toric variety $X(\fan')$ and a birational morphism $\pi: X(\fan')\to X$
such that for each $A\in G$ the composite map $\pi^{-1}\circ f_A\circ \pi$ is an equivariant morphism of $X(\fan')$.
\hfill\qed
\end{prop}

\begin{rem}
Notice that, in the proof of Proposition~\ref{thm_res_fan}, one can make a simplicial refinement $\fan''$ of $\fan'$
such that $\fan''(1)=\fan'(1)$. Then we obtain a projective, simplicial toric variety on which $f_A$ is 1-stable for all $A\in G$.
\end{rem}

\begin{cor}
\label{cor:same_modulus}
Let $X(\fan)$ be a toric variety and $A\in M_n(\Z)\cap GL_n(\Q)$ be diagonalizable.
Assume all the eigenvalues of $A$ are of the same modules, and
$\mu/\bar{\mu}$ is a root of unity for each eigenvalue $\mu$.
Then
\begin{itemize}
\item[(1)]
there exists a projective toric variety $X(\fan')$ and a birational morphism $\pi: X(\fan')\to X$
such that $f_A$ is conjugate to a morphism on $X(\fan')$.
\item[(2)]
there exists a projective, simplicial toric variety $X(\fan'')$ and a birational morphism $\pi: X(\fan'')\to X$
such that the conjugate $\widetilde{f}_A$ is both 1-stable and 2-stable on $X(\fan'')$.
\end{itemize}

\end{cor}

\begin{proof}
For part (1), apply Proposition~\ref{thm_res_morphism} to the monoid generated by $A$.
For (2), pick $\fan''$ to be the fan in the remark after Proposition~\ref{thm_res_fan} for G to be the monoid generated by $A$.
Recall that $A'=|\det(A)|\cdot A^{-1}$, and the fan $\fan''$ satisfies
$-\fan''=\fan''$. Thus, the rays of $\fan''$ still map to rays of $\fan''$, which means
$X(\fan'')$ is 1-stable for both $\widetilde{f}_A$ and $\widetilde{f}_{A'}$. Therefore, $\widetilde{f}_A$ is also 2-stable
on $X(\fan'')$. This concludes the proof of part (2).
\end{proof}

\section{Proof of Theorem~\ref{thm:main_1}}
\label{sec:proof_main_thm}

Under the assumption of Theorem~\ref{thm:main_1}, there are 4 cases:
$|\mu_1|>|\mu_2|>|\mu_3|$, $|\mu_1|=|\mu_2|>|\mu_3|$, $|\mu_1|>|\mu_2|=|\mu_3|$, and $|\mu_1|=|\mu_2|=|\mu_3|$.

The first case is proved in Corollary~\ref{cor:stables}. Indeed, in this case, the refinement $\fan'$ can
be regular and projective. The second case is shown in the discussion after Proposition~\ref{prop:two_real_pos_eval}.
The third case is dual to the second case, thus is also done.

Finally, when we have $|\mu_1|=|\mu_2|=|\mu_3|$, this case is proved in Corollary~\ref{cor:same_modulus}. Therefore,
we proved all cases for Theorem~\ref{thm:main_1}.\hfill\qed

\section{The case with a conjugate pair of eigenvalues whose argument is an irrational multiple of $2\pi$}
\label{sec:complicated_cases}

In the case where there are two complex eigenvalues $\mu,\bar{\mu}$, with $\mu/\bar{\mu}$ not a root of unity,
there are three possible cases. Assume that $\nu$ is the third (real) eigenvalue, then we can have
$|\mu|>|\nu|$, $|\mu|<|\nu|$, or $|\mu|=|\nu|$. The first and the second cases are dual to each other. In fact,
if $A$ is in the first case, the $A'$ will be in the second, and vise versa. Thus, we only need to consider the
first and the third cases.

\subsection{Case 1: $|\mu|>|\nu|$}

For the first case, by the following theorem \cite[Theorem~4.7(2)]{ljl}, we know we cannot make $f_A$ 1-stable.
\begin{thm*}
\label{thm:stable}
Suppose that $A\in \M_n(\Z)$ is a integer matrix.
If $\mu, \bar{\mu}$ are the only eigenvalues of $A$ of maximal modulus, with algebraic multiplicity
one, and if $\mu/\bar{\mu}$ is not a root of unity; then there is no toric birational model
which makes $f_A$ strongly algebraically stable.
\end{thm*}

Next, notice that the 2-stabilization problem for the case $|\mu|>|\nu|$ is equivalent to the 1-stabilization problem
for the case $|\mu|<|\nu|$.

\subsection{Case 2: $|\mu|<|\nu|$}

We also consider the 1-stabilization problem for this case.
First, if we do not start with any given toric variety, then we can certainly find some simplicial toric variety
$X(\fan)$ such that $f_{A}$ is 1-SS on $X(\fan)$ (see \cite[Theorem~4.7(1)]{ljl}). However, the situation is more
complicated when we are given a fixed toric variety $X(\fan)$, and want to stabilize $f_A$ by
refining $\fan$. We need to consider several subcases.

Under the above assumption, let $v$ be an eigenvector associated with $\nu$, and let $\gamma = \R_+ v$.
Let $\Gamma$ be the two-dimensional invariant subspace associated with the eigenvalues $\mu, \bar{\mu}$.
First, we need a lemma.

\begin{lem}
\label{lem:unstable}
Under the above assumption, suppose that $\fan$ satisfies the following two conditions:
\begin{itemize}
\item  $\gamma$ lies on a lower dimensional cone, (i.e., a cone of dimension $\le 2$), and
\item  there is some ray $\gamma_1\in\fan(1)$, $\gamma_1\ne\gamma$, such that $A^k\gamma_1\to\gamma$ as $k\to\infty$
(this means, the angle between $\gamma_1$ and $\gamma$ goes to $0$ as $k\to\infty$).
\end{itemize}
Then, for any refinement $\fan'$ of $\fan$, $f_A$ is not
$1$-stable on $\fan'$, and therefore we cannot stabilize $f_A$ by subdividing $\fan$.
\end{lem}

\begin{proof}
Notice that, for any refinement $\fan'$ of $\fan$, $\gamma$ still lies on a lower dimensional cone.
Moreover, for any refinement $\fan'$, we still have $\gamma_1\in\fan'(1)$ and $A^k\gamma_1\to\gamma$.
Hence, it suffices to show that $f_A$ is not $1$-stable on $X(\fan)$.

Since $\gamma$ is invariant under $A^k$ for all $k$, so under $A$, those two dimensional cones with $\gamma$ as a face
rotate around $\gamma$ with an irrational rotating angle. Thus, none of the three dimensional cone with a face $\gamma$ can be
mapped regularly under all $A^k$. However, since $A^k\gamma_1\to\gamma$, for large $k$, $A^k\gamma_1$ must be in the interior
of some three dimensional cone with a face $\gamma$. Therefore, the condition for $1$-stable cannot be satisfied,
and we conclude that $f_A$ is not $1$-stable on $X(\fan)$.
\end{proof}

Now we are ready to study all the subcases of Case 2.

\begin{enumerate}
\item If both $\gamma$ and $-\gamma$ are in the interior of some three dimensional cones,
and there is no ray of $\fan$ in $\Gamma$.
In this case, if we refine $\fan$ to make both the three-dimensional cones containing $v$ and $-v$,
say $\sigma_1$ and $\sigma_2$, lying on one side of $\Gamma$. Then after certain iterates of $A$, we know
one of the following two things happen for all $\ell$ large enough,
\begin{itemize}
\item $A^\ell(\sigma_1)\subset A^\ell(\sigma_1), A^\ell(\sigma_2)\subset A^\ell(\sigma_2)$, or
\item $A^\ell(\sigma_1)\subset A^\ell(\sigma_2), A^\ell(\sigma_2)\subset A^\ell(\sigma_1)$.
\end{itemize}
Furthermore, every ray in $\fan$ will map into either $\sigma_1$ or $\sigma_2$ after certain iterates. Thus, there
is a refinement $\fan'$ of $\fan$, and an $\ell_0\ge 1$, such that $f_A^\ell$ is 1-stable on $X(\fan')$ whenever
$\ell\ge\ell_0$.
\item If there is some ray of $\fan$ lie on $\Gamma$, then we can look at $A|_\Gamma$ and apply the same argument as in the proof
of \cite[Theorem~4.7(2)]{ljl} to show that $f_A$ cannot be made stable by subdividing $\fan$.

\item If either $\gamma$ or $-\gamma$ lies in the relative interior of a two dimensional cone.
We claim that, in this case, $f_A$ cannot be made stable by subdividing $\fan$.

Without loss of generality, assume $\gamma$ lies in the relative interior of a two dimensional cone $\sigma$.
Suppose that $\gamma_1$ and $\gamma_2$ are the one-dimensional faces of $\sigma$, then since $\sigma$ is
strictly convex, one of $\gamma_1, \gamma_2$ must lie on the same side of $\Gamma$ as $\gamma$, and assume it is $\gamma_1$.
Then we have $A^k\gamma_1\to \gamma$ as $k\to\infty$. By Lemma~\ref{lem:unstable}, we know that $f_A$ cannot be made stable by subdividing $\fan$.

\item If either $\gamma$ or $-\gamma$ is a cone in $\fan$, say $\gamma\in\fan(1)$, and there is another $\gamma_1\in\fan(1)$,
$\gamma_1\ne\gamma$ such that $A^k\gamma_1\to\gamma$ as $k\to\infty$. Then again by Lemma~\ref{lem:unstable},
one cannot make $f_A$ stable by subdividing $\fan$.

\item Finally, if $\gamma\in\fan(1)$, but there is no $\gamma_1\in\fan(1)$, $\gamma_1\ne\gamma$,
with $A^k\gamma_1\to\gamma$ (so we are
not in case 1 or 4), and
$-\gamma$ is in the interior of a three-dimensional cone $\sigma$ (so we are not in case 3).
Moreover, assume that there is no ray of
$\fan$ lies in $\Gamma$ to avoid case 2.

Under the above assumption, $\gamma$ is invariant under $A^k$ for any $k$, and for all $\gamma_1\in\fan(1)$,
$\gamma_1\ne\gamma$, $A^k\gamma_1\to -\gamma$ as $k\to\infty$. Thus for large $k$, $A^k\gamma_1\in\sigma$.
Notice that $\gamma$ cannot be one of the
one-dimensional face of $\sigma$ since $\sigma$ is strictly convex.
Thus the one-dimensional faces of $\sigma$ must also map into $\sigma$ for large $k$. This means for large $k$,
$A^k\sigma\subset\sigma$. To conclude, in this case, we can find a $k_0$ such that $A^k$ is $1$-stable for
$k\ge k_0$.

\end{enumerate}

\begin{ex}
Let $\fan$ be the standard fan for $\P^3$, and
\[
A= \begin{pmatrix} 1 & 1 & 5 \\ 4 & 1 & 2 \\ 1 & 5 & 1 \end{pmatrix}.
\]
The three eigenvalues of $A$ are $7$ and $-2\pm 2\sqrt{2}\i$. Note that the eigenspace associated to
$7$ is spanned by $v=(1,1,1)$. The ray generated by $-v$ is in $\fan$, and is invariant under $A$.
All other rays of $\fan$, which are generated by the three standard basis elements, will tend to $\R_+ v$
under $A^k$. Therefore, we are in Case 2(5). Indeed, $f_A$ is $1$-stable on $\P^3$. On the other hand,
$A'$ is of Case 1, thus $f_A$ is not $2$-stable for any complete toric variety.
\end{ex}

\subsection{Case 3: $|\mu|=|\nu|$}

In this case, we claim that for all complete fan $\fan$
and all $\ell$, $f_A^\ell$ is neither 1-stable nor 2-stable on $X(\fan)$.

Notice that, after a (rational) conjugation, the action of $A$ on cones is to rotate along an axis with an
angle which is irrational modulo $2\pi$. After passing to an iterate of $A$, we can make the angle as
small as we like. However, for any three dimensional polyhedral cone, if we rotate it along any axis for a small angle, it
will not remain staying in any cone.

There is at least one ray $\gamma$ in $\fan$ not in the eigenspace of $\nu$, and for some $k$, $A^k\gamma$ will lie
in the interior of some three dimensional cone. By the argument in the previous paragraph, this three dimensional cone
does not always mapped into another cone. Thus $f_A$ is not $1$-stable in $X(\fan)$.

Finally, $f_{A'}$ will be in this case again, thus cannot be made $1$-stable. This concludes our claim for Case 3.

%%%%%%%%%%%%%%%%%%%%%%%%%%%%%%%%%%%%%%%%%%%%%%%%%%%%%%%%%%%%%%%%%%%%%%
%%%%
%%%%  Bibliography
%%%%
%%%%%%%%%%%%%%%%%%%%%%%%%%%%%%%%%%%%%%%%%%%%%%%%%%%%%%%%%%%%%%%%%%%%%%

\end{document}